\documentclass[11pt,a4paper,twoside,reqno]{amsart}
\usepackage{amsfonts} %
\usepackage{amsmath} %
\usepackage{amssymb} %
\usepackage{amsthm} %
\usepackage{enumerate} %
\usepackage{bbm} %
\usepackage[T1]{fontenc} %
\usepackage[latin9]{inputenc} %
\usepackage{ae,aecompl} %
\usepackage{graphicx} %
\usepackage{nicefrac} %
\usepackage{color} %

\newtheorem{theorem}{Theorem}[section] %
\newtheorem{corollary}[theorem]{Corollary} %
\newtheorem{lemma}[theorem]{Lemma} %
\newtheorem{proposition}[theorem]{Proposition} %
{\theoremstyle{remark} %
  \newtheorem{remark}[theorem]{Remark}} %
{\theoremstyle{definition} %
  \newtheorem{definition}[theorem]{Definition} %
  \newtheorem{example}[theorem]{Example} %
}

\newcommand{\PP}[0]{\ensuremath{\mathbb{P}}}
\newcommand{\CC}[0]{\ensuremath{\mathbb{C}}}
\newcommand{\ZZ}[0]{\ensuremath{\mathbb{Z}}}

\newcommand{\QQ}[0]{\ensuremath{\mathbb{Q}}}

\newcommand{\AAA}[0]{\ensuremath{\mathcal{A}}}

\newcommand{\EE}[0]{\ensuremath{\mathcal{E}}}
\newcommand{\JJ}[0]{\ensuremath{\mathcal{J}}}

\newcommand{\aut}[0]{\ensuremath{\operatorname{Aut}}}
\newcommand{\lin}[0]{\ensuremath{\operatorname{Lin}}}

\newcommand{\diag}[0]{\ensuremath{\operatorname{diag}}}
\newcommand{\PGL}[0]{\ensuremath{\operatorname{PGL}}}
\newcommand{\GL}[0]{\ensuremath{\operatorname{GL}}}

\newcommand{\sing}[0]{\ensuremath{\operatorname{Sing}}}

\begin{document}

\title{On the order of an automorphism of a smooth hypersurface}

\author{V\'\i ctor Gonz\'alez-Aguilera and Alvaro Liendo}
\address{Departamento de Matem\'aticas, Universidad T\'ecnica
  Fe\-de\-ri\-co San\-ta Ma\-r\'\i a, Ca\-si\-lla 110-V, Valpara\'\i
  so, Chile.}  \email{victor.gonzalez@usm.cl}

\address{Mathematisches Institut, Universit\"at Bern, Sidlerstrasse
  5, CH--3012 Bern, Switzerland.}
\email{alvaro.liendo@gmail.com}

\date{\today}

\thanks{
  \mbox{\hspace{11pt}}{\it Key words}: Automorphisms of hypersurfaces,
  linear automorphisms, cyclotomic polynomials, principally
  polarized abelian varieties.\\
  \mbox{\hspace{11pt}}The first author was partially supported by the
  Fondecyt project 1110516 and the Dgip of the UTFSM}

\begin{abstract}
  In this paper we give an effective criterion as to when a positive
  integer $q$ is the order of an automorphism of a smooth hypersurface
  of dimension $n$ and degree $d$, for every $d\geq 3$, $n\geq 2$,
  $(n,d)\neq (2,4)$, and $\gcd(q,d)=\gcd(q,d-1)=1$. This allows us to
  give a complete criterion in the case where $q=p$ is a prime
  number. In particular, we show the following result: If $X$ is a
  smooth hypersurface of dimension $n$ and degree $d$ admitting an
  automorphism of prime order $p$ then $p<(d-1)^{n+1}$; and if
  $p>(d-1)^n$ then $X$ is isomorphic to the Klein hypersurface, $n=2$
  or $n+2$ is prime, and $p=\Phi_{n+2}(1-d)$ where $\Phi_{n+2}$ is the
  $(n+2)$-th cyclotomic polynomial. Finally, we provide some
  applications to intermediate jacobians of Klein hypersurfaces.
\end{abstract}

\maketitle

\section*{Introduction}

Let $n\geq 2$ and $d\geq 3$ be integers. The smooth hypersurfaces of
degree $d$ of the projective space $\PP^{n+1}=\PP^{n+1}(\CC)$ are
classical objects in algebraic geometry. In the following we assume
that $(n,d)\neq (2,4)$.  In this case, the group of regular
automorphisms of any such hypersurface $X$ is finite and equal to the
group of linear automorphisms i.e., every automorphism of $X$ extends
to an automorphism of $\PP^{n+1}$ \cite{mat}. In this paper we study
the possible orders of automorphisms of smooth hypersurfaces.

In Section~\ref{sec-general} we give the following criterion for the
order of an automorphism of a smooth hypersurface: an integer number
$q\in \ZZ_{> 0}$, with $\gcd(q,d)=\gcd(q,d-1)=1$, is the order of
an automorphism of a smooth hypersurface of dimension $n$ and degree
$d$ if an only if
\begin{align}
  \label{res-intro}
  \exists\ell\in\{1,\ldots,n+2\}\quad\mbox{such that}\quad
  (1-d)^\ell\equiv 1\mod q\,.
\end{align}

In Section~\ref{sec-prime} we show that if $q$ is a prime number then
$q$ is the order of an automorphism of a smooth hypersurface of
dimension $n$ and degree $d$ if and only if $q$ divides $d-1$ or
\eqref{res-intro} holds. This is a generalization of our previous
result for cubic hypersurfaces \cite{GL}.

In Section~\ref{sec-klein} we show that if $X$ is a smooth
hypersurface of dimension $n$ and degree $d$ admitting an automorphism
of prime order $p$ then $p<(d-1)^{n+1}$; and if $p>(d-1)^n$ then $X$
is isomorphic to the Klein hypersurface, $n=2$ or $n+2$ is prime, and
$p=\Phi_{n+2}(1-d)$ where $\Phi_{n+2}$ is the $(n+2)$-th cyclotomic
polynomial.

If we restrict to linear automorphism, the results in
Sections~\ref{sec-general}, \ref{sec-prime} and \ref{sec-klein} hold
also in the excluded cases where $n=1$ and $d\geq 3$, or
$(n,d)=(2,4)$.

Finally, in Section~\ref{sec-app} we provide some applications to
principally polarized abelian varieties that are the intermediate
jacobians of Klein hypersurfaces.

\smallskip

The authors are grateful to A.G. Gorinov for pointing us to his
article, to Pierre-Marie Poloni for showing us Zsigmondy's Theorem,
and to the referee for valuable suggestions.

\section{Orders of automorphisms of smooth hypersurfaces of
  $\PP^{n+1}$} \label{sec-general}

Let $n\geq 2$ and $d\geq 3$ be integers, and $(n,d)\neq (2,4)$. In
this section we give a criterion as to when a positive integer $q$,
with $\gcd(q,d)=\gcd(q,d-1)=1$, appears as the order of an
automorphism of some smooth hypersurface of degree $d$ in $\PP^{n+1}$.

Let $V$ be a vector space over $\CC$ of dimension $n+2$, $n\geq2$ with
a fixed basis, and let $\PP^{n+1}=\PP(V)$ be the corresponding
projective space. We also let $\{x_0,\ldots,x_{n+1}\}$ be the dual
basis of the linear forms on $V$ so that $\{x_{i_1}\cdots x_{i_d}\mid
0\leq i_1\leq\ldots\leq i_d \leq n+1\}$ is a basis of the vector space
$S^d(V^*)$ of forms of degree $d$ on $V$.

For a form $F\in S^d(V^*)$, we denote by $X=V(F)\subseteq\PP^{n+1}$
the corresponding hypersurface of dimension $n$ and degree $d$. We
denote by $\aut(X)$ the group of regular automorphisms of $X$ and by
$\lin(X)$ the subgroup of $\aut(X)$ that extends to automorphisms of
$\PP^{n+1}$. Since $(n,d)\neq (2,4)$, by \cite[Theorems 1 and 2]{mat}
if $X$ is smooth then
$$\aut(X)=\lin(X) \quad\mbox{and}\quad |\aut(X)|<\infty.$$

In this setting $\aut(X)<\PGL(V)$ and for any automorphism in
$\aut(X)$ we can choose a representative in $\GL(V)$. This
automorphism induces an automorphism of $S^d(V^*)$ such that
$\varphi(F)=\lambda F$, $\lambda\in \CC^*$. These three automorphisms
will be denoted by the same letter $\varphi$.

In this paper we consider automorphisms of finite order $q$. In this
case, multiplying by an appropriate constant, we can assume that
$\varphi^q=\operatorname{Id}_V$, so that $\varphi$ is also a linear
automorphism of order $q$ of $V$ and $\varphi(F)=\xi^a F$ where $\xi$
is a primitive $q$-th root of unity. Furthermore, we can apply a
linear change of coordinates on $V$ to diagonalize $\varphi$, so that
$$\varphi:V\rightarrow V,\qquad
(\alpha_0,\ldots,\alpha_{n+1})\mapsto(\xi^{\sigma_0}\alpha_0,
\ldots,\xi^{\sigma_{n+1}}\alpha_{n+1}), \qquad 0\leq\sigma_i<q.$$

\begin{definition}
  Letting $\ZZ_q$ be the ring of integers modulo $q$, we define the
  signature $\sigma$ of an automorphism $\varphi$ as above by
  $$\sigma=(\sigma_{0},\ldots,\sigma_{n+1})\in \ZZ_q^{n+2}\,,$$ where we
  identify $\sigma_i$ with its class in the ring $\ZZ_q$. We also
  denote $\varphi=\diag(\sigma)$ and we say that $\varphi$ is a
  diagonal automorphism.
\end{definition}

For every $F\in S^d(V^*)$ and $i\in\{0,\ldots,n+1\}$, we let
$\deg_i(F)$ denote the degree of $F$ seen as a polynomial in
$x_i$. The following simple lemma is a key ingredient in the proof of
Theorem~\ref{resultado}.

\begin{lemma} \label{base} Let $X$ be a hypersurface of dimension $n$
  and degree $d$, given by the homogeneous form $F\in S^d(V^*)$. If
  $\deg_i(F)\leq d-2$, for some $i\in\{0,\ldots,n+1\}$, then $X$ is
  singular.
\end{lemma}
\begin{proof}
  After a linear change of coordinates, we may and will assume
  that $\deg_0(F)\leq d-2$ so that
  $$F=x_0^{d-2}L_2+x_0^{d-3}L_3+\ldots+x_0L_{d-1}+L_d\,,$$
  where $L_j$ is a form of degree $j$ in the variables
  $\{x_1,\ldots,x_{n+1}\}$. Hence,
  \begin{align*}
    &\frac{\partial F}{\partial x_0}=(d-2)x_0^{d-3}L_2+(d-3)x_0^{d-4}L_3+\ldots +L_{d-1},\\
    &\frac{\partial F}{\partial
      x_i}=x_0^{d-2}\frac{dL_2}{dx_i}+x_0^{d-3}\frac{dL_3}{dx_i}+\ldots+\frac{dL_d}{dx_i},
    \quad i\in \{1,\ldots,n+1\}\,.
  \end{align*}

  Now, the Jacobian criterion shows that the point $(1:0:\ldots:0)$ is
  singular.
\end{proof}

The following is the main result of this section.

\begin{theorem} \label{resultado}%
  Let $n\geq 2$ and $d\geq 3$ be integers, and $(n,d)\neq (2,4)$. A
  positive integer $q$, with $\gcd(q,d)=\gcd(q,d-1)=1$ is the order of
  an automorphism of a smooth hypersurface of dimension $n$ and degree
  $d$ if and only if there exists $\ell\in\{1,\ldots,n+2\}$ such that
  $$(1-d)^\ell\equiv 1 \mod q\,.$$
\end{theorem}
\begin{proof}
  To prove the ``only if'' part, suppose that $F\in S^d(V^*)$ is a
  form of degree $d$ such that the hypersurface $X=V(F)\subseteq
  \PP^{n+1}$ is smooth and admits an automorphism $\varphi$ of order
  $q$, with $\gcd(q,d)=\gcd(q,d-1)=1$. Without loss of generality, we
  assume that $\varphi$ is diagonal and we let
  $\sigma=(\sigma_0,\ldots,\sigma_{n+1})\in\ZZ_q^{n+2}$ be its
  signature.

  We have $\varphi(F)=\xi^{a} F$, where $\xi$ is a primitive $q$-th
  root of unity. Let $b$ be such that $d\cdot b \equiv -a \mod q$,
  such a $b$ always exits since $\gcd(q,d)=1$. Consider the
  automorphism $\psi=\xi^{b}\varphi$ of $\GL(V)$. Clearly, $\psi$ and
  $\varphi$ induce the same automorphism in $\PP^{n+1}$. Furthermore,
  for the form $F$ of degree $d$ we have
  $\psi(F)=\xi^{db}\varphi(F)=\xi^{db+a}F=F$. Hence, we may and will
  assume that $\varphi(F)=F$.

  Let now $k_0$ be such that $\sigma_{k_0}\not\equiv 0\mod q$. By
  Lemma \ref{base}, $F$ contains a monomial $x_{k_0}^{d-1}x_{k_1}$ for
  some $k_1\in\{0,\ldots,n+1\}$ (not necessarily with coefficient
  1). The form $F$ is invariant by the diagonal automorphism $\varphi$
  so the monomial $x_{k_0}^{d-1}x_{k_1}$ is also invariant by
  $\varphi$ i.e., $(d-1)\sigma_{k_0}+\sigma_{k_1}\equiv 0 \mod q $,
  and so
  \begin{align}
    \label{equiva}  
    \sigma_{k_1}\equiv(1-d)\sigma_{k_0}\, \mod q\,.
  \end{align}
  Furthermore, since $\gcd(q,d-1)=1$ we have $\sigma_{k_1}\not\equiv 0
  \mod q $, and since $\gcd(q,d)=1$ we have $k_1\neq k_0$.

  Applying the above argument with $k_0$ replaced by $k_1$, we let
  $k_2$ be such that the monomial $x_{k_1}^{d-1}x_{k_2}$ is invariant
  by $\varphi$ and is contained in $F$ (not necessarily with
  coefficient 1). Iterating this process, for all
  $i\in\{3,\ldots,n+2\}$ we let $k_i\in\{0,\ldots,n+1\}$ be such that
  $x_{k_{i-1}}^{d-1}x_{k_i}$ is a monomial in $F$ (not necessarily
  with coefficient 1) invariant by $\varphi$.

  By \eqref{equiva}, we have
  \begin{align}
    \label{equiva2}  
    \sigma_{k_i}\equiv (1-d)\sigma_{k_{i-1}}\equiv (1-d)^2\sigma_{k_{i-2}}
    \equiv(1-d)^i\sigma_{k_0} \mod q ,\ \forall i\in\{2,\ldots,n+2\}\,,
  \end{align}
  and all of the $\sigma_{k_i}$ are non-zero.

  Since $k_i\in\{0,\ldots,n+1\}$ there are at least two
  $i,j\in\{0,\ldots,n+2\}$, $i>j$ such that $k_i=k_j$. Thus
  $\sigma_{k_i}=\sigma_{k_j}$, and since $\sigma_{k_i}\equiv
  (1-d)^i\sigma_{k_0}\mod q$ and $\sigma_{k_j}\equiv
  (1-d)^j\sigma_{k_0} \mod q $, we have
  $$(1-d)^{i-j}\equiv 1 \mod q\,,$$
  and so the ``only if'' part of the theorem follows.

  To prove the converse statement, let $q$ be a positive integer such
  that $\gcd(q,d)=\gcd(q,d-1)=1$, and assume that there exists
  $\ell\in\{1,\ldots,n+2\}$ such that $(1-d)^\ell\equiv 1 \mod q$.

  We let $F\in S^d(V^*)$ be the form 
  $$F=\sum_{i=1}^{\ell-1} x_{i-1}^{d-1}x_i +x_{\ell-1}^{d-1}x_0+
  \sum_{i=\ell}^{n+1}x_i^d\,.$$ %
  By construction, the form $F$ form admits the automorphism
  $\varphi=\diag(\sigma)$, where
  $$\sigma=\big(1,1-d,(1-d)^2,\ldots,(1-d)^{\ell-1},
  \overbrace{0,\ldots,0}^{n+2-\ell\mbox{ times}}\big)\in\ZZ_q^{n+2}\,.$$ %
  An easy modification of the argument in Example~\ref{klein-def}
  below shows that $X=V(F)$ is smooth, proving the theorem.
\end{proof}

\begin{remark} \label{signa} %
  Let $\varphi=\diag(\sigma)$ be an automorphism of order $q$ of the
  smooth hypersurface $X=V(F)$, with $\gcd(q,d)=\gcd(q,d-1)$. As in
  the proof of Theorem \ref{resultado}, we may and will assume that
  $\varphi(F)=F$ and we let $\ell$ be as in Definition \ref{admis}. If
  $\sigma_0\neq 0$ is a component of the signature $\sigma$, then by
  \eqref{equiva2} we have that $(1-d)^i\sigma_0$ is also a component
  of $\sigma$, $\forall i<\ell$. Furthermore, if $q$ is a prime
  number then replacing $\varphi$ by $\varphi^{a}$, where $a\in\ZZ_q$
  is such that $a\cdot\sigma_0\equiv 1\mod q$, we may assume that
  $\sigma_0=1$.
\end{remark}

\section{Automorphisms of prime order of smooth hypersurfaces}
\label{sec-prime}

In this section we study the particular case of automorphisms of prime
order $p$. In this case we are able to give a full characterization of
the prime numbers that appear as the order of an automorphism of some
smooth hypersurface of dimension $n$ and degree $d$. We also show that
the order of such an automorphism is bounded by $(d-1)^{n+1}$.

\begin{definition} \label{admis} We say that a prime number $p$ is
  \emph{admissible in dimension $n$ and degree $d$} if either $p$
  divides $(d-1)$ or there exists $\ell\in\{1,\ldots,n+2\}$ such that
  $$(1-d)^\ell\equiv 1  \mod p\,.$$
\end{definition}

This definition is justified by the following proposition.

\begin{proposition} \label{res-prime} %
  Let $n\geq 2$ and $d\geq 3$ be integers, and $(n,d)\neq (2,4)$. A
  prime number $p$ is the order of an automorphism of a smooth
  hypersurface of dimension $n$ and degree $d$ if and only if $p$ is
  admissible in dimension $n$ and degree $d$.
\end{proposition}
\begin{proof}
  In the case where $p$ does not divide $d$ or $d-1$, the proposition
  follows directly form Theorem~\ref{resultado}. Let $p$ be a prime
  number that divides $d$. The prime number $p$ is admissible with
  $\ell=2$. Indeed, $(1-d)^2= 1-2d+d^2\equiv 1\mod p$.  On the
  other hand, for every $n\geq 2$ and $d\geq 3$, let $X$ be the Fermat
  hypersurface i.e., $X=V(F)$ with
  $$F=x_0^d+x_1^d+\ldots+x_n^d+x_{n+1}^d\,.$$
  The hypersurface $X$ is smooth and admits the automorphism or order
  $d$ given by
  $$\varphi:V\rightarrow V,\qquad(\alpha_0,\alpha_1,\ldots,
  \alpha_{n+1})\mapsto (\xi
  \alpha_0,\alpha_1,\ldots,\alpha_{n+1})\,,$$ where $\xi$ is a
  primitive $d$-th root of unity. Hence, $X$ also admits an
  automorphism of order $p$.

  Let now $p$ be a prime number that divides $d-1$. The prime number
  $p$ is admissible by definition.  On the other hand, for $n\geq 2$
  and $d\geq 3$, let $X=V(F)$ be the hypersurface given by
  $$F=x_0^{d-1}x_1+x_1^d+\ldots+x_n^d+x_{n+1}^d\,.$$
  A routine computation shows that the hypersurface $X$ is smooth and
  admits the automorphism or order $d-1$ given by
  $$\varphi:V\rightarrow V,\qquad(\alpha_0,\alpha_1,\ldots,
  \alpha_{n+1})\mapsto (\xi
  \alpha_0,\alpha_1,\ldots,\alpha_{n+1})\,,$$ where $\xi$ is a
  primitive $(d-1)$-th root of unity. Hence, $X$ also admits an
  automorphism of order $p$.
\end{proof}

\begin{remark}
  Let $n\geq 3$ and $d\geq 3$. It is a trivial consequence of
  Zsigmondy's Theorem \cite{zsigmondy} that there exists at least one
  prime number admissible in dimension $n$ and degree $d$ that is not
  admissible in dimension $n'$ and degree $d$ with $n'<n$. See
  Theorem~1.1 in \cite{steidl} for a modern statement of Zsigmondy's
  Theorem.
\end{remark}

Proposition~\ref{res-prime} allows us to give, in the following
corollary, a first bound for the prime numbers that appear as the
order of an automorphism of some smooth hypersurface. This bound will
be improved in Corollary~\ref{bound2}.

\begin{corollary} \label{bound} %
  Let $n\geq 2$ and $d\geq 3$ be integers, and $(n,d)\neq (2,4)$. If a
  prime number $p$ is the order of an automorphism of a smooth
  hypersurface of dimension $n$ and degree $d$, then $p<(d-1)^{n+1}$.
\end{corollary}
\begin{proof}
  Suppose that $p>(d-1)^{n+1}$. By Proposition~\ref{res-prime}, $p$ is
  admissible in dimension $n$ and degree $d$, and so
  $$(1-d)^{n+1}\equiv 1 \mod p, \quad \mbox{or}\quad (1-d)^{n+2}\equiv 1 \mod p\,.$$
  This yields
  $$p=(1-d)^{n+1}-1, \quad \mbox{or}\quad k\cdot p=(1-d)^{n+2}-1, \quad
  k\in\{1,\ldots, d-1\}\,.$$ %
  Since $1-d\equiv 1\mod d$, we have that $d$ is a divisor of $p$ or
  $k\cdot p$. In both cases, this yields $\gcd(p,d)\neq 1$, which
  provides a contradiction since $p> d$.
\end{proof}

\begin{remark}
  In \cite{szabo} it is shown that the order of a linear automorphism
  of an $n$-dimensional projective variety of degree $d$ is bounded by
  $d^{n+1}$. Hence, in the particular case of prime orders, our bound
  above is already sharper.
\end{remark}

\begin{remark}
  Let $n\geq 2$ and $d\geq 3$ be integers, and $(n,d)\neq (2,4)$. Let
  $X=V(F)$ be a smooth hypersurface of degree $d$ of $\PP^{n+1}$. In
  \cite[Theorem~2]{gorinov} it is shown that the order of $\aut(X)$
  divides
  \begin{align*}
    B=\frac{1}{n+1}\prod_{i=0}^{n}\frac{1}{\binom{n+2}{i}}
    \left((-1)^{n-i+1}+(d-1)^{n-i+2}\right)\cdot\operatorname{lcm}
    \left(\tbinom{n+2}{i}(d-1)^i,(n+2)(d-1)^n)\right)\,.
  \end{align*}

  And in \cite[Page 24, line 4]{gorinov} it is conjectured that every
  prime number $p$ that divides $B$ is the order of an automorphism of
  a smooth hypersurface of dimension $n$ and degree $d$. We can prove
  this conjecture.

  Indeed, if a prime $p$ divides $B$, then $p$ divides
  $\left((-1)^{n-i+1}+(d-1)^{n-i+2}\right)$, $p$ divides $(d-1)$, or
  $p\leq n+2$. In the first two cases $p$ is clearly admissible in
  dimension $n$ and degree $d$. 

  Assume now $p\leq n+2$ and $p$ does not divide $(d-1)$. Since $p$
  does not divide $(d-1)$, we have $(1-d)^\ell\not\equiv 0\mod p$, for
  all $\ell\in \ZZ_{\geq 0}$. Since $p\leq n+2$ there exists
  $\ell,\ell'\in\{1,\ldots n+2\}$, $\ell>\ell'$ such that
  $(1-d)^\ell\equiv(1-d)^{\ell'}\mod p$. Hence
  $(1-d)^{\ell-\ell'}\equiv 1\mod p$ and so $p$ is admissible in
  dimension $n$ and degree $d$. Now the conjecture follows from
  Proposition~\ref{res-prime}.
\end{remark}

The criterion in Theorem \ref{resultado} is easily computable. As en
example, in Table~\ref{table1} we give all the admissible prime
numbers in degree $4$ for different values of the dimension $n$.
\begin{table}[!ht]
  \begin{tabular}{ | c | c | }
    \hline	
    $n$ & admissible primes \\
    \hline		
    3 & $2, 3, 5, 7, 61$  \\
    4 & $2, 3, 5, 7, 13, 61$ \\
    5 & $2, 3, 5, 7, 13, 61, 547$ \\
    6 & $2, 3, 5, 7, 13, 41, 61, 547$ \\
    7 & $2, 3, 5, 7, 13, 19, 37, 41, 61, 547$ \\
    8 & $2, 3, 5, 7, 11, 13, 19, 37, 41, 61, 547$ \\
    9 & $2, 3  5, 7, 11, 13, 19, 37, 41, 61, 67, 547, 661$ \\
    10 & $2, 3, 5, 7, 11, 13, 19, 37, 41, 61, 67, 73, 547, 661$ \\
    \hline
  \end{tabular} \vspace{1ex}
  \caption{Admissible prime numbers in degree $4$ for $3\leq n\leq
    10$.}
  \label{table1}
\end{table}

In Table~\ref{table2} we give the maximal admissible prime number
$p$ for small $n$ and $d$.
\begin{table}[!ht]
  \begin{tabular}{ | c || c | c | c | c | c | c | c | c | c | c | }
    \hline	
    $n\backslash d$ & 3 & 4 & 5 & 6 & 7 & 8 & 9 \\
    \hline\hline
    2 & 5 & - & 17 & 13 & 37 & 43 & 19 \\
    \hline
    3 & 11 & 61 & 41 & 521 & 101 & 191 & 331 \\
    \hline
    4 & 11 & 61 & 41 & 521 & 101 & 191 & 331 \\
    \hline
    5 & 43 & 547 & 113 & 521 & 197 & 911 & 5419 \\
    \hline
    6 & 43 & 547 & 257 & 521 & 1297 & 1201 & 5419 \\
    \hline
    7 & 43 & 547 & 257 & 5167 & 46441 & 117307 & 87211 \\
    \hline
    8 & 43 & 547 & 257 & 5167 & 46441 & 117307 & 87211 \\
    \hline
    9 & 683 & 661 & 2113 & 5281 & 51828151 & 10746341 & 87211 \\
    \hline
  \end{tabular}\vspace{1ex}
  \caption{Maximal admissible prime for $2\leq n\leq 9$, $3\leq
    d\leq 9$, and $(n,d)\neq(2,4)$.}
  \label{table2}
\end{table}

\section[Hypersurfaces with an automorphism of prime order $p\geq
(d-1)^n$ ]{Smooth hypersurfaces admitting an automorphism \\ of prime
  order $p>(d-1)^n$}
\label{sec-klein}

In this section we show the following result: A smooth hypersurface
$X$ of dimension $n$ and degree $d$ admits an automorphism of prime
order $p>(d-1)^n$ if and only if $X$ is the Klein hypersurface, $n=2$
or $n+2$ is prime, and $p=\Phi_{n+2}(1-d)$, where $\Phi_{n+2}$ is the
$(n+2)$-th cyclotomic polynomial. First, we recall some results about
cyclotomic polynomials, see \cite[Ch. VI, \S 3]{lang} for proofs.

\begin{definition}
  For every $m\in\ZZ_{>0}$, the $m$-th cyclotomic polynomial is
  defined as
  $$\Phi_m(t)=\prod_{\xi}(t-\xi)\,,$$
  where the product is over all primitive $m$-th roots of unity $\xi$.
\end{definition}

It is well known that $\Phi_m(t)$ is irreducible over $\QQ$ and has
integer coefficients. Furthermore, a routine computation shows that
$\Phi_1(t)=t-1$ and for every $q$ prime and $r\geq 1$
$$\Phi_q(t)=t^{q-1}+t^{q-2}+\ldots +1, \quad\mbox{and}\quad
\Phi_{q^r}(t)=\Phi_q(t^{q^{r-1}})\,.$$

The main result about cyclotomic polynomials that we will need in the
sequel is the following factorization
$$t^n-1=\prod_{r|n} \Phi_r(t)\,,$$
where $r|n$ means $r$ is a divisor of $n$.

Our next theorem gives a criterion for the existence of a smooth
hypersurface of dimension $n$ and degree $d$ admitting an automorphism
of prime order $p>(d-1)^n$. For the proof we need the following simple
inequalities.
\begin{align*}
  &\Phi_1(1-d)=-d,\quad \Phi_2(1-d)=2-d,\quad
  \Phi_4(1-d)=d^2-2d+2>(d-1)^2,
  \quad\mbox{and} \\
  &(d-1)^{q-2}<\Phi_q(1-d)=(1-d)^{q-1}+\ldots
  +1<(d-1)^{q-1},\quad\mbox{for all}\quad q\geq 3\mbox{ prime}\,.
\end{align*}

\begin{lemma} \label{pmax} %
  Let $n\geq 2$, $d\geq 3$ be integers, and $(n,d)\neq (2,4)$. There
  exists a smooth hypersurface of dimension $n$ and degree $d$
  admitting an automorphism of prime order $p>(d-1)^n$ if and only if
  $n=2$ or $n+2$ is prime, and $p=\Phi_{n+2}(1-d)$.
\end{lemma}
\begin{proof}
  We prove first the ``only if'' part of the lemma. Assume that there
  exists a smooth hypersurface of degree $n$ and dimension $d$
  admitting an automorphism of prime order $p>(d-1)^n$. By
  Proposition~\ref{res-prime}, $p$ is admissible in dimension $n$ and
  degree $d$ and by Corollary~\ref{bound}, $p$ is not admissible in
  dimension $n-1$ and degree $d$. Hence $(1-d)^{n+2}\equiv 1\mod p$
  and so
  $$(1-d)^{n+2}-1= k\cdot p,\quad\mbox{for some}\quad
  k\in\{-(d-1)^2,\ldots,(d-1)^2\}\,.$$

  If $n=2$ then
  $$(1-d)^{n+2}-1=\Phi_1(1-d)\cdot\Phi_2(1-d)\cdot\Phi_4(1-d)\,,$$
  and the only possibility is $k=\Phi_1(1-d)\Phi_2(1-d)=d(d-2)$
  and $p=\Phi_4(1-d)=d^2-2d+2$.

  If $n+2$ is prime, then
  $$(1-d)^{n+2}-1=\Phi_1(1-d)\cdot\Phi_{n+2}(1-d)\,,$$ %
  and the only possibility is $k=\Phi_1(1-d)=-d$ and
  $p=\Phi_{n+2}(1-d)=(1-d)^{n+1}+\ldots+1$.

  To finish this direction of the proof, we have to show that these
  two are the only possible cases. If $n\neq 2$ and $n+2$ is not
  prime, then
  $$n+2=q\cdot n',\quad\mbox{or}\quad n+2=2^i\,,$$
  where $q\geq 3$ is a prime number,$n'\geq 2$, and $i\geq 3$.
  
  Assume fist that $n+2=q\cdot n'$.  In this case
  $$(1-d)^{n+2}-1=\Phi_1(1-d)\cdot\Phi_{q}(1-d)\cdot P(1-d)\,,$$ %
  for some polynomial $P(t)$. Let
  $k'=\Phi_1(1-d)\cdot\Phi_{q}(1-d)$. Since
  $k'<d(d-1)^{q-1}<p$, $k$ is a multiple of $k'$. But
  $k'>d(d-1)^{q-1}>(d-1)^2$ which provides a contradiction.

  Finally, assume that $n+2=2^i$, $i\geq 3$.  In this case
  $$(1-d)^{n+2}-1=\Phi_{4}(1-d)\cdot P(1-d)\,,$$ %
  for some polynomial $P(t)$. Let $k'=\Phi_{4}(1-d)=(d-1)^2+1$. Since
  $k'<(d-1)^3<p$, $k$ is a multiple of $k'$. but $k'>(d-1)^2$
  which provides a contradiction.
  \medskip

  To prove the ``if'' part, let $n=2$ or $n+2$ be prime, and assume
  that $\Phi_{n+2}(1-d)$ is prime. In both cases
  $\Phi_{n+2}(1-d)\geq (1-d)^n$.  If $n+2$ is prime, then
  $$(1-d)^{n+2}-1= \Phi_{1}(1-d) \cdot \Phi_{n+2}(1-d)\equiv 0 \mod
  \Phi_{n+2}(1-d)\,,$$ and so $\Phi_{n+2}(1-d)$ is admissible in
  dimension $n$ and degree $d$.
  If $n=2$, then
  $$(1-d)^4-1=\Phi_{1}(1-d) \cdot\Phi_{2}(1-d) \cdot \Phi_{4}(1-d) \equiv 0 \mod
  \Phi_{4}(1-d),.$$ %
  Hence, $\Phi_{4}(1-d)$ is admissible in dimension $2$ and degree
  $d$. This completes the proof.
\end{proof}

In the following corollary, that follows directly form 
Lemma~\ref{pmax}, we give a sharp bound for the order of an automorphism of a
smooth hypersurface of dimension $n$ and degree $d$.

\begin{corollary} \label{bound2}%
  Let $n\geq 2$ and $d\geq 3$ be integers, and $(n,d)\neq
  (2,4)$. Assume that a smooth hypersurface of dimension $n$ and
  degree $d$ admits an automorphism of prime order $p$.
  \begin{enumerate}[(i)]
  \item If $n=2$ or $n+2$ is prime, and $\Phi_{n+2}(1-d)$ is prime,
    then $p\leq \Phi_{n+2}(1-d)$. This bound is sharp.
  \item In any other case, $p<(d-1)^n$.
  \end{enumerate}
\end{corollary}

\begin{remark}\label{infinite}
  \begin{enumerate}[$(i)$]
  \item The condition in Lemma~\ref{pmax} that $\Phi_{n+2}(1-d)$ is
    prime is fulfilled, for instance, in the cases where $(n,d)$ is
    $(2,3)$, $(2,5)$, $(2,7)$, $(3,3)$, $(3,4)$, $(3,6)$,
    $(5,3)$,$(5,4)$, $(9,3)$, and $(9,7)$. See Table~\ref{table2}.

  \item Assume that $(n,d)$ is such that $\Phi_{n+2}(1-d)$ is prime
    and $n\neq 2$. Then
    $$p=\frac{(1-d)^{n+2}-1}{(1-d)-1}\,.$$
    Prime numbers of this form are usually known as generalized
    Mersenne primes or repunit primes. For $d=-1$ they correspond to
    the classical Mersenne primes and for $d=3$ they are usually
    called Wagstaff primes. It is conjectured that there are
    infinitely many such primes \cite{williams,melham}.
  \end{enumerate}
\end{remark}

In the following example we define the classical Klein hypersurfaces
that will be the subject of the remaining of this section.

\begin{example} \label{klein-def}%
  For any $n\geq 1$ and $d\geq 2$, we define the Klein hypersurface of
  dimension $n$ and degree $d$ as $X=V(F)\in\PP^{n+1}$, where
  \begin{align} \label{klein-eq}
    F=x_0^{d-1}x_1+x_1^{d-1}x_2+\ldots+x_n^{d-1}x_{n+1}+
    x_{n+1}^{d-1}x_0\,.
  \end{align}
  It is well known that $X$ is smooth except in the case where $d=2$
  and $n\equiv 2\mod 4$. In lack of a good reference we provide a
  short argument.
\end{example}

\begin{proof}
  Assume that $\alpha=(\alpha_0:\ldots:\alpha_{n+1})\in X$ is a singular
  point i.e.,
  $$F(\alpha)=0,\quad\mbox{and}\quad \frac{\partial F}{\partial x_i}(\alpha)=0\,.$$
  It is clear from the equations $\frac{\partial F}{\partial
    x_i}(\alpha)=0$ that $\alpha_i\neq 0$, for all $i\in
  \{0,\ldots,n+1\}$. Furthermore, the equations $x_i\frac{\partial
    F}{\partial x_i}(\alpha)=0$ imply
  \begin{align*}    
    \alpha_i^{d-1}\alpha_{i+1}&=(1-d)\alpha_{i+1}^{d-1}\alpha_{i+2}=
    (1-d)^2\alpha_{i+2}^{d-1}\alpha_{i+3}=
    \ldots=(1-d)^{n-i}\alpha_{n}^{d-1}\alpha_{n+1}\\
    &=(1-d)^{n-i+1}\alpha_{n+1}^{d-1}\alpha_{0}\,.
  \end{align*}
  Hence,
  $$F(\alpha)=R\cdot \alpha_{n+1}\alpha_0,\quad\mbox{where}\quad
  R=\sum_{i=0}^{n+1}(1-d)^i\,.$$ %
  If $d\neq 2$, then $R\neq 0$ and so $F(\alpha)=0$ implies
  $\alpha_0=0$ or $\alpha_{n+1}=0$ which provides a contradiction. In
  the case where $d=2$ a routine computation shows that the quadratic
  form $F$ is singular if and only if $n\equiv 2\mod 4$.
\end{proof}

This kind of hypersurfaces were first introduced by Klein who studied
the automorphism group of the Klein hypersurface of dimensions $1,3$
and degree $3$ \cite{klein}. For the proof of the theorem below, we
need the following simple lemma that follows from the uniqueness of
the decomposition of an integer in base $(1-d)$.

\begin{lemma} \label{li} %
  Let $d\geq 3$ and $a_i\in\{1,\ldots,d-2\}$, $0\leq i\leq n+1$. If
  $\sum_i a_i(1-d)^i=0$, then $a_i=0$ for all $i$.
\end{lemma}

The following is the main result of this section.

\begin{theorem} \label{kuni} %
  Let $n\geq 2$ and $d\geq 3$ be integers, and $(n,d)\neq (2,4)$. A
  smooth hypersurface $X=V(F)$ of dimension $n$ and degree $d$ admits
  an automorphism $\varphi$ of prime order $p>(d-1)^n$ if and only if
  $X$ is isomorphic to the Klein hypersurface, $n=2$ or $n+2$ is
  prime, and $p=\Phi_{n+2}(1-d)$.
\end{theorem}
\begin{proof}
  Since $p>(d-1)^n$, by Corollary~\ref{bound}, $p$ is not admissible
  in dimension $n-1$ and degree $d$. Hence, by Remark \ref{signa}, we
  can assume that $\varphi(F)=F$ and $\varphi=\diag(\sigma)$, where
  $$\sigma=(\sigma_0,\ldots,\sigma_{n+1})=
  \big(1,(1-d),(1-d)^2,\ldots,(1-d)^{n+1}\big)\,.$$

  The Klein hypersurface defined by the form in \eqref{klein-eq}
  admits the automorphism $\varphi$ above. This together with
  Lemma~\ref{pmax} proves the ``if'' part.
  
  Assume now that $X=V(F)$ is a smooth hypersurface of dimension $n$
  and degree $d$ admitting the automorphism $\varphi$ of prime order
  $p>(d-1)^n$. Let $\EE\subset S^d(V^*)$ be the eigenspace associated
  to the eigenvalue 1 of the linear automorphism
  $\varphi:S^d(V^*)\rightarrow S^d(V^*)$, so that $F\in\EE$. In the
  following we compute a basis for $\EE$. Let $\mathbf{x}^\alpha$ be a
  monomial in $S^d(V^*)$ i.e.
  $$\mathbf{x}^\alpha:=x_0^{\alpha_0}\cdots
  x_{n+1}^{\alpha_{n+1}},\quad \sum_{i=0}^{n+1} \alpha_i=d,\mbox{ and
  }\alpha_i\geq 0\,.$$
  
  We have 
  $$\mathbf{x}^\alpha\in \EE\Leftrightarrow L:=
  \alpha_0+\alpha_1(1-d)+\ldots+\alpha_{n+1}(1-d)^{n+1}\equiv 0 \mod
  p\,.$$ Since $\alpha_{n+1}=d-\sum_{i=0}^n\alpha_i$, we have
  \begin{align*}
    L&=d(1-d)^{n+1}+\sum_{i=0}^n \alpha_i\left((1-d)^i-(1-d)^{n+1}\right)\\
     &=d(1-d)^{n+1}+d\cdot\sum_{i=0}^n \alpha_i\left((1-d)^i+\ldots+(1-d)^{n}\right)\,.
  \end{align*}
  Letting $\beta_i=\sum_{j=0}^n \alpha_j$, for all $0\leq i\leq n$, we
  have and $0\leq \beta_i\leq \beta_j\leq d$, for all $i<j$, and
  $$L=d\cdot M,\quad\mbox{where}\quad M=\beta_0+\beta_1(1-d)+
  \ldots+\beta_{n}(1-d)^{n}+(1-d)^{n+1}\,.$$ %
  Since $d$ is invertible in $\ZZ_p$ we have
  $$\mathbf{x}^\alpha\in\EE \Leftrightarrow L\equiv 0 \mod p\Leftrightarrow M\equiv 0 \mod p\,.$$
  
  By Lemma~\ref{pmax} we know that $p=\Phi_{n+2}(1-d)$ and $n=2$ or
  $n+2$ is prime. We divide the proof in two cases. \medskip

  \noindent \textbf{Case $\pmb{n+2}$ is prime:} In this case
  $p=1+(1-d)+\ldots+(1-d)^{n+1}$.  If $\beta_n<d-1$ then $M=p$ and
  Lemma~\ref{li} shows that $\beta_i=1$, $\forall i$. This corresponds
  to $\mathbf{x}^\alpha=x_{n+1}^{d-1}x_0$.

  If $\beta_n=d-1$ then
  $M=0$ and Lemma~\ref{li} shows that $\beta_i=0$, $\forall i<n$. This
  corresponds to $\mathbf{x}^\alpha=x_{n}^{d-1}x_{n+1}$.

  If $\beta_j=d$, $\forall j>k+1$ and $\beta_k<d$, for some $k<n$ then
  $$M\equiv \beta_0+\ldots+\beta_k(1-d)^k+(1-d)^{k+1}\mod p$$
  This gives $\beta_k=(d-1)$ and $\beta_i=0$, for all $i<k$. This
  corresponds to $\mathbf{x}^\alpha=x_{k}^{d-1}x_{k+1}$.
  
  Hence, $\EE=\langle 
  x_{n+1}^{d-1}x_0,x_k^{d-1}x_{k+1}; 0\leq k\leq n\rangle$ and
  $$F=a_0\cdot x_0^{d-1}x_1+a_1\cdot x_1^{d-1}x_2+\ldots+a_n
  \cdot x_n^{d-1}x_{n+1}+a_{n+1}\cdot x_{n+1}^{d-1}x_0\,.$$ %
  Since $X=V(F)$ is smooth, by Lemma~\ref{base} $a_i\neq 0$, $\forall
  i$ and applying a linear change of coordinates we can put
  $$F=x_0^{d-1}x_1+x_1^{d-1}x_2+x_2^{d-1}x_3+x_3^{d-1}x_0\,.$$

  \noindent \textbf{Case $\pmb{n=2}$:} In this case $p=(d-1)^2+1$ and
  so $M=(\beta_0-\beta_2)+(\beta_1-1)(1-d)$. If $M=0$ then
  $\beta_0=\beta_1=\beta_2=1$, or $\beta_0=\beta_1=0$ and
  $\beta_2=d-1$. This corresponds to $\mathbf{x}^\alpha=x_3^{d-1}x_0$
  and $\mathbf{x}^\alpha=x_2^{d-1}x_3$, respectively.

  If $M=p$ then $\beta_0=d-1$ and $\beta_1=\beta_2=d$, or $\beta_0=0$,
  $\beta_1=d-1$ and $\beta_2=d$. This corresponds to
  $\mathbf{x}^\alpha=x_0^{d-1}x_1$ and
  $\mathbf{x}^\alpha=x_1^{d-1}x_2$, respectively.

  Hence, $\EE=\langle
  x_0^{d-1}x_1,x_1^{d-1}x_2,x_2^{d-1}x_3,x_3^{d-1}x_0\rangle$ and
  $$F=a_0\cdot x_0^{d-1}x_1+a_1\cdot x_1^{d-1}x_2+
  a_2\cdot x_2^{d-1}x_3+a_3\cdot x_3^{d-1}x_0\,.$$
  With the same argument as above, we can apply a linear change of
  coordinates to put
  \begin{align*}
    F=x_0^{d-1}x_1+x_1^{d-1}x_2+x_2^{d-1}x_3+x_3^{d-1}x_0\,.\tag*{\qedhere}
  \end{align*}
\end{proof}

Let now $(n,d)$ be a pair satisfying the condition of
Theorem~\ref{kuni} and let $\varphi$ be the automorphism of order
$p=\Phi_{n+2}(1-d)$ of the Klein hypersurface $X=V(F)$. In the
remaining of this section, we study the geometry of the action of the
cyclic group $\langle\varphi\rangle\simeq \ZZ/p\ZZ$ on $X$.

Recall first that $\varphi=\diag(\sigma)$, where $\sigma=
\left(1,(1-d),(1-d)^2,\ldots,(1-d)^{n+1}\right)$. Since
$\sigma_i\not\equiv\sigma_j\mod p$ for all $i\neq j$, the only fixed
points of $\varphi$ are the images of the $n+2$ standard basis
vectors of $V=\CC^{n+2}$ in $\PP^{n+1}$.

For our next result, we say that a cyclic quotient singularity is of
type $\tfrac{1}{p}(a_1,\ldots,a_n)$ if it is locally isomorphic to the
singularity at the vertex of the quotient of $\CC^{n}$ by $\ZZ/p\ZZ$,
where the $(\ZZ/p\ZZ)$-action is generated by the automorphism
$$(x_1,\ldots x_n)\mapsto (\xi^{a_1}x_1,\ldots,\xi^{a_n}x_n),
\quad\mbox{with}\quad \xi^p=1\mbox{ and } \xi\neq 1\,.$$

\begin{proposition} \label{singular}
  Let $n,d,p$ be as in Theorem~\ref{kuni}. The quotient
  $Y=X/(\ZZ/p\ZZ)$ of the Klein hypersurface by the cyclic group
  generated by $\varphi$ has $n+2$ singular points of singularity type
  $\tfrac{1}{p}\left((1-d)^2-1,\ldots,(1-d)^{n+1}-1\right)$.
\end{proposition}

\begin{proof}
  The set of singular points of $Y$ is contained in the image under
  the quotient map of the set of fixed points of the
  $(\ZZ/p\ZZ)$-action on $X$ given by $\varphi$. Furthermore, the
  Klein hypersurface $X$ admits the automorphism
  $$\psi:\PP^{n+1}\rightarrow \PP^{n+1},\quad (x_0:\ldots:
  x_{n+1})\mapsto (x_1:\ldots :x_{n+1}:x_0)\,,$$ %
  and since the orbit of the fixed point $\alpha=(1:\ldots:0)$ by
  $\langle\psi\rangle$ is the $n+2$ fixed points, the singularity type
  of all the singular points of $Y$ is the same.

  To compute the singularity type of the image of the point $\alpha$ in
  $Y$, we pass to the invariant affine open set $U=\{x_0\neq
  1\}\simeq\CC^{n+1}$ in $\PP^{n+1}$ with coordinates $x_1,\ldots
  x_{n+1}$. Now, the fixed point $\alpha$ corresponds to $\bar{0}\in U$,
  the Klein hypersurface $X|_U$ is given by the equation
  $$x_1+x_1^{d-1}x_2+\ldots+x_n^{d-1}x_{n+1}+x_{n+1}^{d-1}=0\,, $$
  and the automorphism $\varphi|_U$ is given by
  $$\varphi|_U=\diag\left((1-d)-1,(1-d)^2-1,\ldots,(1-d)^{n+1}-1\right)\,.$$

  Since $\alpha\in X|_U$ is a smooth point, the singularity type of
  the image of $\alpha$ in the quotient, is the same as the
  singularity type at the image of the origin of the quotient of the
  tangent space $T_\alpha X$ by the linear action
  $\widetilde{\varphi}$ induced by $\varphi$. The tangent space
  $T_\alpha X$ is given by $\{x\in \CC^{n+1}\mid x_1=0\}\simeq\CC^n$
  and
  $$\widetilde{\varphi}=\diag\left((1-d)^2-1,\ldots,(1-d)^{n+1}-1\right)\,.$$
  Hence, the singularity type of the image of $\alpha$ is
  $\tfrac{1}{p}\left((1-d)^2-1,\ldots,(1-d)^{n+1}-1\right)$. To
  complete the proof, we remark that $\widetilde{\varphi}$ is not a
  pseudo-reflection and so the image of $p$ is indeed a singular
  point.
\end{proof}

A cyclic quotient singularity is always Cohen-Macaulay and rational
but not necessarily Gorenstein. By \cite{watanabe}, a cyclic quotient
singularity $\CC^n/G$ is Gorenstein if and only if the acting group
$G$ is a subgroup of $\operatorname{SL}(\CC^n)$. In the following
corollary, we apply this result to prove that the singularities of $Y$
are never Gorenstein.

\begin{corollary}
  Let $n,d,p$ be as in Theorem~\ref{kuni}. Then the singularities of
  $Y=X/(\ZZ/p\ZZ)$ are not Gorenstein.
\end{corollary}
\begin{proof}
  Since the singular points of $Y$ are of type
  $\tfrac{1}{p}\left((1-d)^2-1,\ldots,(1-d)^{n+1}-1\right)$, we only
  have to show that the automorphism of $\CC^n$ given by
  $\diag\left((1-d)^2-1,\ldots,(1-d)^{n+1}-1\right)$ does not belong
  to $\operatorname{SL}(\CC^n)$. This happens if and only if
  $$(1-d)^2-1+\ldots+(1-d)^{n+1}-1\not\equiv 0\mod p\,.$$
  Since $p=\Phi_{n-2}(1-d)$, we have
  $(1-d)^2-1+\ldots+(1-d)^{n+1}-1=p+d-n-2\,.$ and so $Y$ has
  Gorenstein singularities if and only if $d=n+2$. Finally, a
  straightforward computation shows that if $d=n+2$, then $n$ divides
  $\Phi_{n+2}(1-d)=p$, which provides a contradiction.
\end{proof}

\begin{remark}
  The condition $p>(d-1)^n$ coming from Theorem~\ref{kuni} in the
  above corollary is essential. Indeed, if we let $n=3$ and $d=5$,
  then the Klein hypersurface $X$ admits an automorphism $\varphi$ of
  order $p=41<(d-1)^n=64$ with the same signature $\sigma$ as in
  Theorem~\ref{kuni}, $$\sigma=
  \left(1,(1-d),(1-d)^2,\ldots,(1-d)^{n+1}\right)=(1,37,16,18,10)\,.$$

  By the proof Theorem~\ref{kuni}, the singular points of quotient of
  $X$ by $\langle\varphi\rangle$ are of type
  $\tfrac{1}{p}\left((1-d)^2-1,\ldots,(1-d)^{n+1}-1\right)=
  \tfrac{1}{41}(15,17,9)$. Now, this quotient singularity is
  Gorenstein since $15+17+9=41\equiv 0\mod 41$.
\end{remark}

\section{An application to intermediate jacobians of Klein
  hypersurfaces}
\label{sec-app}

Letting $(n,d)$ satisfy Theorem~\ref{kuni} we let $X$ be the Klein
hypersurface of dimension $n$ and degree $d$. Let also $\varphi$ be
the automorphism of $X$ of prime order $\Phi_{n+2}(1-d)$. In
\cite{shioda} Shioda studies the action of $\varphi$ on the Lieberman
Jacobian of $X$. In particular, the author proves that the spectrum of
the action of $\varphi$ in $H^n(X,\CC)$ is composed by all the
primitive roots of unity with multiplicity one. We remark that the
``naive'' question \cite[Section 3 (b)]{shioda} is not so naive by our
Remark~\ref{infinite} (ii). Here, we study in more detail the
particular cases where the intermediate jacobian $\JJ(X)$ admits a
principal polarization.

Let $X$ be a smooth hypersurface of degree $d$ of $\PP^{n+1}$. It is
known \cite[Page 22]{deligne}, that the intermediate jacobian $\JJ(X)$
is a non trivial principally polarized abelian variety (p.p.a.v.) if
and only if $n=1$ and $d\geq 3$, $n=3$ and $d=3,4$, or $n=5$ and
$d=3$. In this case, the dimension of the cohomology $H^n(X,\CC)$ is
given by \cite{voisin}
\begin{align*}
  \dim H^n(X,\CC)=\frac{(d-1)^{n+2}+(-1)^n(d-1)}{d}\,.
\end{align*}

It is also possible using the residue calculus of Griffiths to give an
explicit representation of $H^n(X,\CC)$. Indeed, let $X=V(F)$
where $F$ is a form of degree $d$ and let $S^{l}=H^0({\mathbb
  P}^{n+1},{\mathcal O}_{{\mathbb P}^{n+1}}(l))$, for all $l\geq 0$ so
that $S=\bigoplus_l S^l$ is the polynomial ring in $n+2$ variables. We
also let $J_F=\bigoplus_l J_F^l$ the homogeneous jacobian ideal of $F$
generated by the partial derivatives $\frac{\partial F}{\partial
  x_i}$, and $R_F^l=S^l/J_F^l$ be the $l$-component of the jacobian
ring $R_F=S/J_F$. With all this definitions, we have
\cite{voisin}
\begin{align} \label{victor-eq} %
  H^{n+1-r,r-1}(X, \CC) \simeq R_F^{rd-n-2}\,.
\end{align}

\medskip

Let now $X$ be the Klein hypersurface given by $V(F)$, where $F$ is as
in \eqref{klein-eq}. In the following, we will study the
p.p.a.v. $\JJ(X)$ in the case where $(n,d)=(3,3)$, $(n,d)=(3,4)$, and
$(n,d)=(5,3)$. Letting $k=\frac{n-1}{2}$ we let $\iota$ be the
canonical injection $\iota:H^n(X,\ZZ)\hookrightarrow H^{k,k+1}(X,\CC)$
so that
$$\JJ(X)=H^{k,k+1}(X,\CC)/\iota(H^n(X,\ZZ))\,.$$

By Theorem~\ref{kuni}, $X$ admits an automorphism $\varphi$ of order
$p=11$, $p=61$, and $p=43$, respectively. The automorphism $\varphi$
preserves the Hodge structure of $X$ and since this Hodge structure
is of level $1$, $\varphi$ induces an automorphism
$\widetilde{\varphi}$ of the p.p.a.v. $\mathcal J(X)$.

Let $\AAA_g$ be the moduli space of p.p.a.v. of dimension $g$ and let
$\sing\AAA_g$ its singular locus. If a prime number $p$ is the order
of an automorphism $A\in \AAA_g$, then $p\leq 2g+1$. In the case where
$p=2g+1$, the p.p.a.v. $A$ is called extremal and corresponds to
$0$-dimensional irreducible components of $\sing\AAA_g$ in the sense
of \cite{victor}. As we will see below, $\mathcal J(X)$ is extremal.

\subsection{The cubic Klein threefold}
This is the case where $(n,d)=(3,3)$. In this case $\mathcal J(X)$ is
p.p.a.v of dimension $5$ and since $\widetilde{\varphi}$ is of order
$11$, $\mathcal J(X)$ is indeed extremal.

By \eqref{victor-eq} for $r=3$ we have $H^{1,2}(X,\CC) \simeq R_F^4$,
and a basis for $H^{1,2}(X,\CC)$ via this isomorphism is given
by
$$\{x_1x_2x_3x_4,x_0x_2x_3x_4,x_0x_1x_3x_4,x_0x_1x_2x_4,
x_0x_1x_2x_3 \}\,.$$

The automorphism $\varphi$ of order $p=11$ of the Klein hypersurface
is given by $\diag(1,9,4,3,5)$. Hence, the spectrum of the induced
isomorphism $\overline{\varphi}: T_0\JJ(X)\simeq H^{1,2}(X,\CC)
\rightarrow T_0\JJ(X)$ is given by $C=\{{\xi}^{10}, {\xi}^2, {\xi}^7,
{\xi}^8, {\xi}^6 \}$. Since $C\cap \overline{C}=\emptyset$ and $C\cup
\overline{C}$ corresponds to all the primitive $11$-th roots of unity,
$\JJ(X)$ is p.p.a.v. of complex multiplication type \cite{bennama} and
is a $0$-dimensional component of the singular locus of ${\mathcal
  A}_5$ \cite{victor}.

Furthermore $\psi(a)=a^{5}$ stabilize the complex multiplication type
$C$ of $\JJ(X)$ and induce a permutation of coordinates on $T_0\JJ(X)$
given by $\sigma=(10,6,8,7,2)$. Thus, $\JJ(X)$ is contained in a
component of $\sing\AAA_{5}$ corresponding to p.p.a.v. admitting the
automorphisms $\sigma$ of order $5$. Let us denote this component by
$\AAA_{5}(5,{\sigma})$.

The spectrum of $\sigma$ is $\{{\zeta}^0,{\zeta}^1 ,{\zeta}^2,
{\zeta}^3,{\zeta}^4 \}$, where $\zeta$ is a primitive $5$-th root of
unity. Then it follows from \cite[Page 1]{victor} that
$\AAA_{5}(5,{\sigma})$ is a $3$-dimensional subvariety of $\sing
\AAA_{5}$ that contain $\JJ(X)$.

\subsection{The cubic Klein fivefold}
This is the case where $(n,d)=(5,3)$. In this case $\mathcal J(X)$ is
p.p.a.v of dimension $21$ and since $\widetilde{\varphi}$ is of order
$43$, $\mathcal J(X)$ is indeed extremal.

By \eqref{victor-eq} for $r=3$ we have $H^{2,3}(X,\CC) \simeq R_F^2$,
and a basis for $H^{2,3}(X,\CC)$ via this isomorphism is given by
$\{x_ix_j\in S^2(V^{*}) \mid 0\leq i<j\leq 6\}$. The automorphism
$\varphi$ of order $p=43$ of the Klein hypersurface $X$ is given by
$$\varphi=\diag(1,41,4,35,16,11,21)\,.$$ %
Hence, the spectrum of the induced isomorphism $\overline{\varphi}:
T_0\JJ(X)\simeq H^{2,3}(X,\CC) \rightarrow T_0\JJ(X)$ is given by
$$C=\{{\xi}^2, {\xi}^3, {\xi}^5, {\xi}^8, {\xi}^9, {\xi}^{12},
{\xi}^{13}, {\xi}^{14},{\xi}^{15}, {\xi}^{17}, {\xi}^{19}, {\xi}^{20},
{\xi}^{22}, {\xi}^{25}, {\xi}^{27}, {\xi}^{32},{\xi}^{33}, {\xi}^{36},
{\xi}^{37}, {\xi}^{39}, {\xi}^{42} \}\,.$$%
Since $C\cap
\overline{C}=\emptyset$ and $C\cup \overline{C}$ corresponds to all
the primitive $43$-th roots of unity, $\JJ(X)$ is p.p.a.v. of complex
multiplication type \cite{bennama} and is a $0$-dimensional component
of the singular locus of $\AAA_{21}$ \cite{victor}.

Furthermore $\psi(a)=a^{11}$ stabilize the complex multiplication type
$C$ of $\JJ(X)$ and induce a permutation of coordinates on $T_0\JJ(X)$
of order $7$ given by
$$\sigma=(2,22,27,39,42,32,8) (3,33,19,37,20,5,12)(9,13,14,25,17,15,36)$$
Thus, $\JJ(X)$ is contained in a component of $\sing\AAA_{21}$
corresponding to p.p.a.v. admitting the automorphisms $\sigma$ of
order $7$. Let us denote this component by $\AAA_{21}(7,{\sigma})$.

The spectrum of $\sigma$ is 
$$\{\zeta^0,\zeta^0,\zeta^0,\zeta^1,\zeta^1,\zeta^1,\zeta^2,\zeta^2,
\zeta^2,\zeta^3,\zeta^3,\zeta^3,\zeta^4,\zeta^4,\zeta^4,\zeta^5,
\zeta^5,\zeta^5,\zeta^6,\zeta^6,\zeta^6\}\,,$$ %
where $\zeta$ is a primitive $7$-th root of unity. Then it follows
from \cite[Page 1]{victor} that $\AAA_{21}(7,{\sigma})$ is a
$33$-dimensional subvariety of $\sing \AAA_{21}$ that contain $\JJ(X)$.

\subsection{The quartic Klein threefold}

This is the case where $(n,d)=(3,4)$. In this case $\mathcal J(X)$ is
p.p.a.v of dimension $30$ and since $\widetilde{\varphi}$ is of order
$61$, $\mathcal J(X)$ is indeed extremal.

By \eqref{victor-eq} for $r=3$ we have $H^{1,2}(X,\CC) \simeq R_F^3$,
and a basis for $H^{1,2}(X,\CC)$ via this isomorphism is given by
$$\{x_ix_jx_k\in S^3(V^{*}) \mid 0\leq i\leq j\leq k\leq 4,
\mbox{ and } i\neq k\}\,.$$

The automorphism $\varphi$ of order $p=61$ of the Klein hypersurface
$X$ is given by $\varphi=\diag(1,58,9,34,20)$.  Hence, the spectrum of
the induced isomorphism $$\overline{\varphi}: T_0\JJ(X)\simeq
H^{1,2}(X,\CC) \rightarrow T_0\JJ(X)$$ is given by
\begin{align*}
  \{&\xi^{60}, \xi^{11}, \xi^{36}, \xi^{22}, \xi^{56}, \xi^{7},
  \xi^{32}, \xi^{18},\xi^{19}, \xi^{44}, \xi^{30},
  \xi^{8}, \xi^{55}, \xi^{41}, \xi^{3}  \\
  &\xi^{28},\xi^{14}, \xi^{15}, \xi^{40}, \xi^{26}, \xi^{4}, \xi^{51},
  \xi^{37},\xi^{52}, \xi^{38}, \xi^{16}, \xi^{2}, \xi^{49}, \xi^{27},
  \xi^{13} \}\,.
\end{align*}
Since $C\cap \overline{C}=\emptyset$ and $C\cup \overline{C}$
corresponds to all the primitive $61$-th roots of unity, $\JJ(X)$ is
p.p.a.v. of complex multiplication type \cite{bennama} and is a
$0$-dimensional component of the singular locus of $\AAA_{30}$
\cite{victor}.

Furthermore $\psi(a)=a^{9}$ stabilize the complex multiplication type
$C$ of $\JJ(X)$ and induce a permutation of coordinates on $T_0\JJ(X)$
of order $5$ given by
\begin{align*}
\sigma=(&2,18,40,55,7)(3,27,60,52,41)(4,36,19,49,14) \\
(&8,11,38,37,28)(13,56,16,22,15)(26,51,32,44,30)\,.
\end{align*}
Thus, $\JJ(X)$ is contained in a component of $\sing\AAA_{30}$
corresponding to p.p.a.v. admitting the automorphisms $\sigma$ of
order $5$. Let us denote this component by $\AAA_{30}(5,{\sigma})$.

The spectrum of $\sigma$ is 
$$\{\overbrace{\zeta^0,\ldots,\zeta^0}^{\mbox{6 times}},
\overbrace{\zeta^1,\ldots,\zeta^1}^{\mbox{6 times}},
\overbrace{\zeta^2,\ldots,\zeta^2}^{\mbox{6 times}},
\overbrace{\zeta^3,\ldots,\zeta^3}^{\mbox{6 times}},
\overbrace{\zeta^4,\ldots,\zeta^4}^{\mbox{6 times}}\}\,,$$ %
where $\zeta$ is a primitive $5$-th root of unity. Then it follows
from \cite[Page 1]{victor} that $\AAA_{30}(5,{\sigma})$ is a
$93$-dimensional subvariety of $\sing \AAA_{30}$ that contain
$\JJ(X)$.

\end{document}